\documentclass{amsart}
\usepackage{amssymb,latexsym}
\usepackage{amsmath}
\usepackage{amscd}
\usepackage{graphicx}
 \usepackage{color}
\usepackage{enumerate}
\numberwithin{equation}{section}
\theoremstyle{plain}
 \newtheorem{theorem}{Theorem}[section]
 \newtheorem{lemma}[theorem]{Lemma}

\theoremstyle{definition}
 
 \newtheorem{remark}[theorem]{Remark}

\newenvironment{enumeratei}{\begin{enumerate}[\quad\upshape (i)]} {\end{enumerate}}

\newcommand \datum {July 23, 2017, extended version} % search for ALL of its occurrences !!!
\renewcommand \phi {\varphi}
\renewcommand \epsilon {{\boldsymbol\varepsilon}}
\newcommand \vvecr {\vec r\kern 1.5pt'}
\newcommand \ucirc {C_{\kern-1pt\textup{unit}}}
\newcommand \bnd [1] {\partial #1} %boundary
\newcommand \lne {\ell}
\newcommand \flne[1] {{\ell^-_{#1}}}
\newcommand \llne[1] {{\ell^+_{#1}}}
\newcommand \fll[1] {{\ell_{#1}}}

\newcommand \dirangle[2] {\angle(#1,#2)^\to}

\newcommand \slcr [1] {\textup{Sli}(#1)}

\newcommand \dir [1] {\textup{dir}(#1)}
\newcommand \nothing [1] {}
\newcommand \phullo {\textup{Conv}} % Real HULL Operator in  R^#1
\newcommand\set [1]{\{#1\}}
\newcommand \tuple [1] {\langle #1 \rangle}
\newcommand \pair [2] {\tuple{#1,#2}}
\newcommand \real {\mathbb R}
\newcommand \plreal {{\real^+}}
\newcommand \preal {\real^{2}}
\newcommand \red [1] {{\color{red}#1\color{black}}}
\newcommand \tbf[1] {\textbf{#1}} 
\newcommand \dist[2]{\textup{dist}(#1,#2)}

\newcommand \fmidp[1] {\textup{midp}}

\begin{document}
\title[Characterizing  circles]
{Characterizing circles by a convex combinatorial property}

\author[G.\ Cz\'edli]{G\'abor Cz\'edli}
\email{czedli@math.u-szeged.hu}
\urladdr{http://www.math.u-szeged.hu/\textasciitilde{}czedli/}
\address{University of Szeged\\ Bolyai Institute\\Szeged,
Aradi v\'ertan\'uk tere 1\\ Hungary 6720}

\thanks{This research was supported by
NFSR of Hungary (OTKA), grant number K 115518}

\begin{abstract} 
Let $K_0$ be a compact convex subset of the plane $\preal$, and assume that $K_1\subseteq \preal$ is similar to $K_0$, that is, $K_1$ is the image of $K_0$ with respect to a similarity transformation $\preal\to\preal$.
Kira Adaricheva and Madina Bolat have recently proved that if $K_0$ is a disk and both $K_0$ and $K_1$ are included in a triangle with vertices $A_0$, $A_1$, and $A_2$, then there exist a $j\in \set{0,1,2}$ and a $k\in\set{0,1}$ such that $K_{1-k}$ is included in the convex hull of $K_k\cup(\set{A_0,A_1, A_2}\setminus\set{A_j})$. Here we prove that this property characterizes disks among compact convex subsets of the plane. In fact, we prove even more since we replace ``similar'' by 
``isometric'' (also called ``congruent'').  Circles are the boundaries of disks, so our result also gives a characterization of circles. 
%Finally, a variant with $A_0$, $A_1$, and $A_2$ being sets similar to $K_0$ rather than points are also given.  
\end{abstract}

%52C (1991-now) Discrete geometry 
%52C99 (1991-now) None of the above, but in this section 
%52A01 (1980-now) Axiomatic and generalized convexity  

\dedicatory{Dedicated to the seventy-fifth birthday (in 2018) of L\'aszl\'o Hatvani }


\date{\red{\datum}}

\maketitle

\section{Aim and introduction}
\subsection*{Our goal}
The real plane and the usual convex hull operator on it will be denoted by $\preal$ and $\phullo$, respectively. That is, for a set $X\subseteq \preal$ of points,  $\phullo(X)$ is the smallest convex subset of $\preal$ that contains $X$. 
As usual, if $X$ and $Y$ are subsets of $\preal$ such that $Y=\phi(X)$ for a similarity transformation or an isometric transformation $\phi\colon\preal\to\preal$, then $X$ and $Y$ are \emph{similar} or \emph{isometric} (also called \emph{congruent}), respectively. \emph{Disks} are convex hulls of circles and \emph{circles} are boundaries of disks.  The singleton subsets of $\preal$ are both disks and circles. A \emph{compact} subset of $\preal$ is a  topologically closed and bounded subset. 
Our aim is to  prove the  following theorem.

\begin{theorem}\label{thmmain} If $K_0$ is a compact convex subset of the plane $\preal$, then the following three conditions are equivalent.
\begin{enumeratei}
\item\label{thmmaina} $K_0$ is a disk.
\item\label{thmmainb} For every  $K_1\subseteq \preal$ and for arbitrary points $A_0,A_1,A_2\in \preal$, if $K_1$ is similar to $K_0$ and both $K_0$ and $K_1$ are contained in the triangle $\phullo(\set{A_0,A_1,A_2})$, then there exist a $j\in \set{0,1,2}$ and a $k\in\set{0,1}$ such that $K_{1-k}$ is contained in $\phullo\bigl(K_k\cup(\set{A_0,A_1, A_2}\setminus\set{A_j})\bigr)$.
\item\label{thmmainc} The same as the second condition but ``similar'' is replaced by
``isometric''.
\end{enumeratei}
\end{theorem}

Our main achievement is that  \eqref{thmmainc} implies \eqref{thmmaina}. The implication   \eqref{thmmaina}  $\Rightarrow$ \eqref{thmmainb} was discovered and proved by  Adaricheva and Bolat \cite{kaczg}; see also Cz\'edli~\cite{czgabp} for a shorter proof. The implication 
\eqref{thmmainb} $\Rightarrow$ \eqref{thmmainc} is trivial.

\subsection{Prerequisites and motivation}
This paper is self-contained for most mathematicians. Interestingly enough, besides abstract convex geometry, the present work is motivated mainly by lattice theory.
For more about the background and motivation of this topic, the reader may want, but need not, to see, for example, 
Adaricheva~\cite{adaricarousel}, Adaricheva and Cz\'edli~\cite{kaczg},  Adaricheva and Nation~\cite {kirajbbooksection} and \cite{kajbn}, Cz\'edli~\cite{czgcoord}, \cite{czgcircles}, and \cite{czgabp}, 
Cz\'edli and Kincses~\cite{czgkj}, Edelman and Jamison~\cite{edeljam}, Kashiwabara,  Nakamura, and  Okamoto~\cite{kashiwabaraatalshelling}, Monjardet~\cite{monjardet}, and   Richter and Rogers~\cite{richterrogers}. 
Note that the property described in \ref{thmmain}\eqref{thmmainb} is called the ``Weak Carousel property'' in Adaricheva 
and Bolat~\cite{kabolat}. The motivation discussed above explains that a wide readership is targeted; geometers would need less details at some parts of the paper.

\subsection{Outline} In the rest of the paper, we prove Theorem~\ref{thmmain}. 
\footnote{\red{The first two versions of this paper are at  \texttt{http://arxiv.org/abs/1611.09331}, this is the third version (of \datum), and see the author's website  for possible updates, if any.}}

\section{Notation and terminology} If $X$ is a point and $Y$ is a line or another point, then their distance will be denoted by $\dist X Y$. 
For points $X_1,X_2\in\preal$, the closed line segment between $X_1$ to $X_2$ will be denoted by $[X_1,X_2]$. In this subsection, $H\subseteq\preal$ will denote a compact convex set. Its  \emph{boundary}  will be denoted by $\bnd H$. For a line $\lne$, if $H$ is contained in one of the closed halfplanes or in one of the open halfplanes determined by $\lne$, then we say that $H$ \emph{lies on one side} or \emph{lies strictly on one side} of $H$, respectively. If $\lne\cap H\neq\emptyset$ and $H$ lies in one of the halfplanes determined by $\lne$, then $\lne$ is  a \emph{supporting line} of $H$; in this case, $\lne\cap H\subseteq \bnd H$. 
The properties of supporting lines that we need here are more or less clear by geometric intuition;  see Cz\'edli and Stach\'o~\cite{czgstacho} for easily available details and proofs, or see Bonnesen and Fenchel~\cite{bonnesenfenchel} for a more advanced treatise.
A \emph{direction} is a point $\alpha$ on the
\begin{equation}
\text{\emph{unit circle}\quad $\ucirc:=\set{\pair x y\in\preal: x^2+y^2=1}$.}
\label{equnitCrlcle}
\end{equation}
Unless otherwise stated explicitly, we always assume that our lines, typically the supporting lines, are \emph{directed}; we denote the direction of such a line  $\lne$ by $\dir \lne\in \ucirc$. 
If $\lne$ is a supporting line of $H$, then 
\begin{equation}
\text{$\lne$ is always directed so  that $H$ is on its  left.}
\label{eqtxtlwSnthGt}
\end{equation}
Note that 
\begin{equation}
\parbox{6cm}{for each $\alpha\in\ucirc$, there is a unique supporting line $\lne$ such that $\dir\lne=\alpha$.}
\label{eqtxtuNspalpha}
\end{equation}
A \emph{secant} of $H$ is a line that passes through an interior point of $H$. We know from Yaglom and Boltyanski$\breve\i$~\cite[1-4 in page 7]{yaglomboltyanskii} that 
\begin{equation}
\text{
a secant intersects $\bnd H$ in exactly two points.}
\label{eqtxtSccsKtt}
\end{equation} 
Related to \eqref{eqtxtuNspalpha}, we formulate the following statement for later reference.

\begin{figure}[ht] 
\centerline
{\includegraphics[scale=1.0]{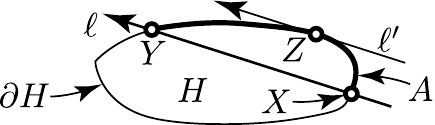}}
\caption{Illustration for Lemma~\ref{lemmaeqtxtdRlLe}
\label{figMjRrT}}
\end{figure}%

\begin{lemma}\label{lemmaeqtxtdRlLe}
Let $\lne$ be a secant of $H$ with $\set{X,Y}=\lne\cap\bnd H$ such that the arc $A$ of $\bnd H$ from $X$ to $Y$ going forward (that is, counterclockwise) is on the right of $\lne$; see Figure~\ref{figMjRrT}. Then $A$ has a unique last point $Z$ such that the line $\lne'$ through $Z$ with $\dir{\lne'}=\dir\lne$ is a supporting line of $H$.  Furthermore, all points of $A\setminus\set X$ that are not after $Z$ are strictly on the right of $\lne$. In particular, $Z$ and $\lne'$ are strictly on the right of $\lne$.
\end{lemma}

For $P\in\bnd H$, the first and last supporting lines through $P$, with respect to counterclockwise rotation, are the \emph{first semitangent} and the \emph{last semitangent} of $H$ through $P$,  respectively. If there is only one supporting line through $P$, then it is called the \emph{tangent line} through  $P$ (or at $P$).
\begin{equation}
\parbox{10.5cm}{For each $P\in\bnd H$, the first and last semitangent through $P$ uniquely exist; they will be denoted by $\flne P$ and $\llne P$, respectively.  When they coincide,  $\fll P:=\flne P=\llne P$ stands for the \emph{tangent line} through $P$.}
\label{eqtxtfllllnE}
\end{equation}

\begin{figure}[ht] 
\centerline
{\includegraphics[scale=1.0]{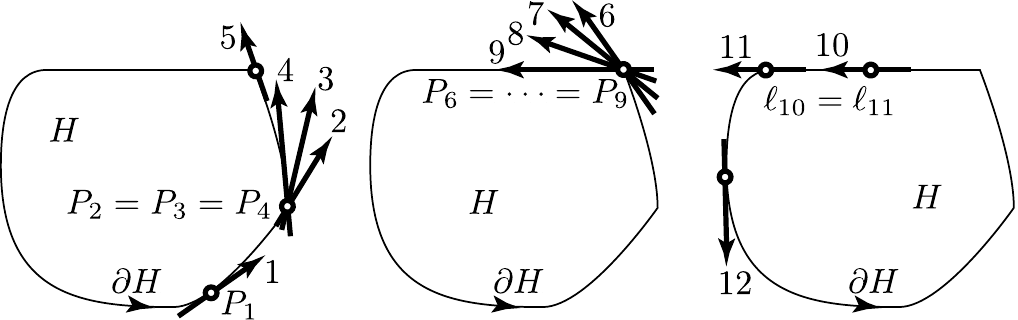}}
\caption{Pointed supporting lines
\label{figglide}}
\end{figure}%

By a \emph{pointed supporting line} we mean a pair $\pair P\lne$ such that $P\in \bnd H$ and $\lne$ is a supporting line of $H$ through $P$. As Figure~\ref{figglide}  shows, none of $P$ and $\lne$ determines the other in general. 
When we transform $\pair P\lne$ to another pointed supporting line continuously by sliding $P$ along $\bnd H$, or turn $\lne$, or doing both, then we \emph{slide-turn} $\pair P\lne$.
It is proved in Cz\'edli and Stach\'o~\cite{czgstacho} that a pointed supported line  
\begin{equation}
\parbox{11cm}{
$\pair P\lne$ can be \emph{slide-turned} continuously  around $H$ forward.}
\label{eqtxtslDtRn}
\end{equation}
In  Figure~\ref{figglide}, the supporting lines $\pair{P_1}{\lne_1}$,  $\pair{P_2}{\lne_2}$, \dots,  $\pair{P_{12}}{\lne_{12}}$, denoted simply by their subscripts, are consecutive snapshots of this slide-turning. %Slide-turning backward is also possible. 

\begin{figure}[ht] 
\centerline
{\includegraphics[scale=1.0]{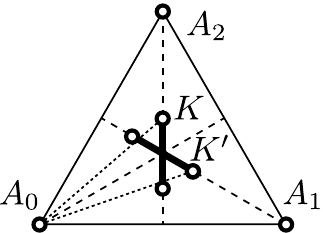}}
\caption{Excluding that $K$ is a line segment
\label{figrtQmsMgl}}
\end{figure}%

\section{Lemmas and proofs}\label{sectcompconv}
\begin{proof}[Proof of Lemma~\ref{lemmaeqtxtdRlLe}]
Applying \eqref{eqtxtuNspalpha} to the intersection of $H$ and the closed halfplane on the right of $\lne$, we obtain a supporting line $\lne'$  of $H$ such that  $\dir{\lne'}=\dir\lne$. Since the secant $\lne$ contains an interior point of $H$ by definition, $\lne'$ is strictly on the right of $\lne$. So is the last point, $Z$, of the closed line segment $\lne'\cap\bnd H=\lne'\cap H$. Suppose, for a contradiction, that the arc $A\setminus\set X$ has a point $P$  not after $Z$ such that $P$ is not strictly on the right of $\lne$. Since $\bnd H$ is known to be a rectifiable Jordan curve, see Yaglom and Boltyanski$\breve\i$~\cite{yaglomboltyanskii} or the survey part of  Cz\'edli and Stach\'o~\cite{czgstacho},  we conclude by continuity that there is a point $P'\in\lne\cap A \subseteq \lne\cap \bnd H$ between $P$ and $Z$. 
Since $\bnd H$ cannot have multiple points and $X, P, P', Z,$ and $Y$ follow in this counterclockwise order, possibly with $P=P'$, we obtain that $P'\neq Y$. Thus,  $X$, $Y$, and $P'$ are three distinct points in $\lne\cap \bnd H$, 
contradicting \eqref{eqtxtSccsKtt}.
\end{proof}

As already mentioned, 
we need to prove only the implication  \eqref{thmmainc} $\Rightarrow$ \eqref{thmmaina} for Theorem~\ref{thmmain}. 
 Since every singleton subset of $\preal$ is a disk, we assume that $K:=K_0$ is not a singleton. 
 Figure~\ref{figrtQmsMgl}, where the triangle is regular and $K':=K_1$, shows that if  $K:=K_0$ is a non-singleton line segment, then it fails to satisfy \ref{thmmain}\eqref{thmmainc}.  Clearly, if $K$ is not a line segment, then it contains three non-collinear points, whereby its interior is nonempty. 
Thus, in the rest of the paper, we assume that 
\begin{equation}
\parbox{7.0cm}{$K:=K_0\subseteq \preal$ is a compact set with non\-empty interior and it satisfies \ref{thmmain}\eqref{thmmainc}.}
\label{eqtxtdhgrhcVStVn}
\end{equation}
We need to prove that $K$ is a disk. In our  figures that follow, 
$K$ will be the  grey-filled set while $K'$ and $K^\ast$ will be isometric copies of $K$. 
For directed lines $\lne$ and $\lne'$, the 
\emph{directed angle} from $\lne$ to $\lne'$, denoted by $\dirangle \lne {\lne'}$,
is the unique $\alpha\in [0,2\pi)$ such that rotating $\lne$ counterclockwise by $\alpha$, we obtain a line of direction $\dir{\lne'}$. Usually, $\dirangle \lne {\lne'}\neq \dirangle {\lne'} {\lne}$. 
The following lemma is illustrated in Figure~\ref{figrtRngl} twice.

\begin{lemma}[Intersection Lemma]\label{lemmaintersect} Assuming \eqref{eqtxtdhgrhcVStVn},  let  $K'$ be isometric to $K$ and let $P$ be an intersection point of $K$ and $K'$. Assume that both $K$ and $K'$ have tangent lines through $P$, see \eqref{eqtxtfllllnE}, and let $\lne_P$ and $\lne'_P$ denote these unique tangent lines, respectively. If $\alpha:=\dirangle{\lne_P}{\lne_P'}$  is in the open interval $(0,\pi)$, then there is a common supporting line $\lne$ of $K$ and $K'$ such that $0<\dirangle{\lne_P}\lne<\alpha$, the first point $P^\dagger$ in $\lne\cap K'$ precedes the last point $P^\ddagger$ in  $\lne\cap K$, $P^\dagger\notin K$, and  $P^\ddagger\notin K'$.
\end{lemma}

\begin{figure}[ht] 
\centerline
{\includegraphics[scale=1.0]{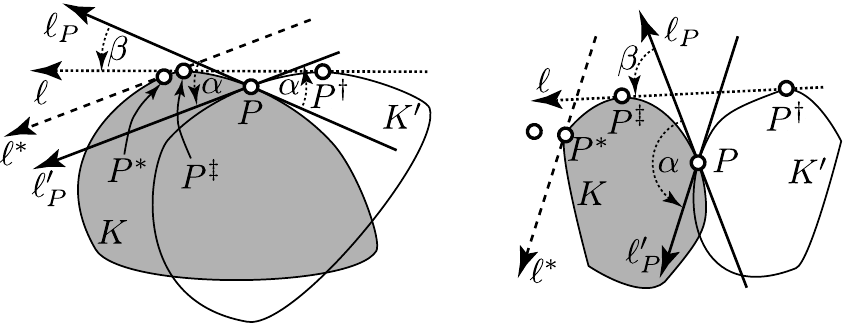}}
\caption{Illustration for Lemma~\ref{lemmaintersect}
\label{figrtRngl}}
\end{figure}%

\begin{proof}
Since $\alpha\in(0,\pi)$ and the tangent line $\lne_P$ is the unique supporting line of $K$ through $P$, see Figure~\ref{figrtRngl}, we obtain that $\lne_{P'}\neq\lne_{P}$ and $K$ is not on the left of $\lne_{P'}$. If we turn $\lne_P$ around $P$ a very little backward (clockwise) to get a line $\lne^\circ$ with $\dirangle {\lne^\circ}{\lne'_P}\in (0,\pi)$, then 
$\lne^\circ$ is not a supporting line of $K$ 
by the above-mentioned uniqueness. This implies that $K$ is not on the right of  $\lne'_P$. Thus, $\lne'_P$ is a secant of $K$. Hence, by Lemma~\ref{lemmaeqtxtdRlLe}, $K$ has a supporting line $\lne^\ast$ together with a unique last point $P^\ast\in \lne^\ast\cap \bnd K$ such that $\dir{\lne^\ast}=\dir{\lne_P'}$ and both $\dir{\lne^\ast}$ and $P^\ast$ are strictly on the right of $\lne_P'$. 
Since $K'$ is on the left of $\lne'_P$, $K'$ is strictly on the left of $\lne^\ast$.
Now, $\pair {P}{{\lne_P}}$ and $\pair {P^\ast}{{\lne^\ast}}$
are  pointed supporting lines of $K$, and we know that $K'$ is not on the left of $\lne_P$ but it is on the left of $\lne^\ast$.   Hence, when we are slowly slide-turning  $\pair {P}{\lne_P}$ to $\pair {P^\ast}{\lne}$ forward, see \eqref{eqtxtslDtRn}, we obtain a first pointed supporting line $\pair {P_0^\ddagger}{\lne}$ of $K$ such that $K'$ is on the left of $\lne$. By continuity, $\lne$ is a supporting line also of $K'$. Hence, it is a common supporting line of $K$ and $K'$. The intersection $\lne\cap\bnd K$ is a closed interval from $P_0^\ddagger$ to its other endpoint, which we denote by $P^\ddagger$. Note that $P^\ddagger$ can coincide with $P_0^\ddagger$. Note also that   $\pair {P^\ddagger}{\lne}$ is still a pointed supporting line of $K$, it is obtained from $\pair {P_0^\ddagger}{\lne}$ by slide-turning it forward, and $P^\ddagger$ is the last point on $\lne$ that belongs to $K$. We have that $0<\dirangle{\lne_P}\lne<\alpha$, since slide-turning forward changes the direction forward. Since slide-turning forward also changes the  point components forward,  $P^\ddagger$ is after $P$ but not after $P^\ast$. Since  $\lne_P$ is the only supporting line of $K$ through $P$ but, being distinct from $\lne_P'$, it is not a common supporting line, we have that $P^\ddagger\neq P$. Hence, by Lemma~\ref{lemmaeqtxtdRlLe}, $P^\ddagger$ is strictly on the right of $\lne'_P$, whereby $P^\ddagger\notin K'$. The slide-turning procedure that yielded $\lne$ makes it clear that
\begin{equation}
\parbox{8.1cm}{with direction in the open interval $(\dir{\lne_P},\dir{\lne_P'})$, $\lne$ is the only common tangent line of $K$ and $K'$.}
\label{eqtxtsdkhGmdl}
\end{equation}
Interchanging $\langle K,\lne_P,$forward$\rangle$ and 
 $\langle K',\lne'_{P},$backward$\rangle$,  in other words, by the \emph{left-right dual} of the argument above,
  we  obtain a common tangent line $\lne'$ of $K$ and $K'$ and a unique first point $P^\dagger\in\lne'\cap K'$. 
 It follows from \eqref{eqtxtsdkhGmdl} that $\lne=\lne'$. This completes the proof of Lemma~\ref{lemmaintersect}.
\end{proof}

\begin{lemma}[Cross Lemma]\label{lemmacross} Assuming \eqref{eqtxtdhgrhcVStVn},  let  $K'$ be isometric to $K$ and let  $t_1\neq t_2$ be common supporting lines of $K$ and $K'$. Then it is impossible that for each $i\in\set{1,2}$, the first point $U_i$ of $(K\cup K')\cap t_i$ is in $K\setminus K'$ while the last point
$U'_i$ of $(K\cup K')\cap t_i$ is in $K'\setminus K$;  here ``first'' and ``last'' are understood in the sense of \eqref{eqtxtlwSnthGt}.
\end{lemma}

The name ``Cross Lemma'' comes from the visual idea that the excluded situation means that $K$ and $K'$ cross each other; see Figures~\ref{figprfrtcrsL},  \ref{figcTngLcrL}, and  both parts of Figure~\ref{figrtcrsSll}.  Let us agree that $K$ and $K'$ \emph{cross each other} if the prohibited situation described in Lemma~\ref{lemmacross} holds for $\pair K{K'}$ or $\pair{K'}K$; the solid lines and the dashed lines on the right of Figure~\ref{figrtcrsSll} indicate that both cases can simultaneously happen. 
The Cross Lemma says that if $K'$ is isometric to $K$, then $K$ and $K'$ cannot cross each other.

\begin{figure}[ht] 
\centerline
{\includegraphics[scale=1.0]{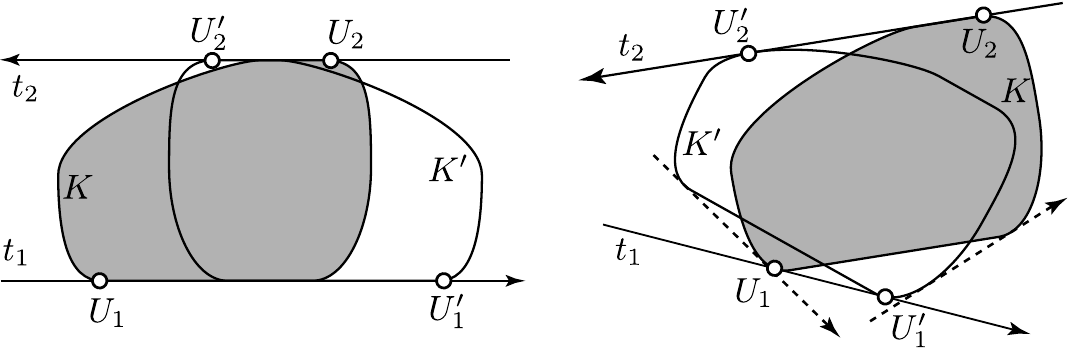}}
\caption{Illustration for the Cross Lemma
\label{figrtcrsSll}}
\end{figure}%

\begin{proof}[Proof of Lemma~\ref{lemmacross}]
Suppose, for a contradiction, that 
$K$ and $K'$ cross each other;  see Figures~\ref{figprfrtcrsL} and \ref{figcTngLcrL}, which exemplify different angles formed by $t_1$ and $t_2$.

\begin{figure}[ht] 
\centerline
{\includegraphics[scale=1.0]{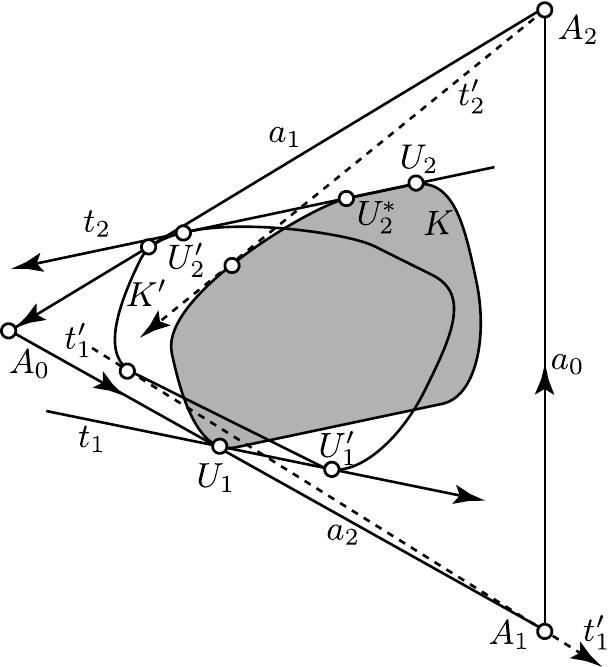}}
\caption{Illustration for the proof of the Cross Lemma
\label{figprfrtcrsL}}
\end{figure}%

First, assume that $\dir{t_1}=\dir{t_2}+\alpha$ for some $\alpha\in (0,\pi]$, where the addition is understood modulo $2\pi$. This means that going forward along $t_2$ from $U_2'$ and  backwards along $t_1$ from $U_1$, we reach their intersection point. This point is indicated neither in Figure~\ref{figprfrtcrsL}, because it would be too far on the left, nor in Figure~\ref{figcTngLcrL}, because it would make the figure too crowded. Note that ``left'' is the direction of $(\dir{t_2}+(\pi+\dir{t_1}))/2$, which is the direction of the bisector of the directed angle from $t_2$ to $-t_1$. While ``left'' is more or less faithfully represented in Figure~\ref{figprfrtcrsL}, 
we have turned Figure~\ref{figcTngLcrL} clockwise 
to make it better fit the page.

Slide-turn the dashed supporting line $\pair{U_2}{t_2}$  around $K$  forward so that the direction changes but only to a very small extent; the supporting line (component of the pointed supporting line of $K$) that we obtain in this way is denoted by $t_2'$; see the upper dashed line in Figures~\ref{figprfrtcrsL} and \ref{figcTngLcrL}.
Similarly, we obtain the lower dashed supporting line $t_1'$ of $K'$ by slide-turning $\pair{U_1'}{t_1}$ around $K'$ backwards; again, the difference between $\dir{t'_1}$ and $\dir{t_1}$ should be very little.  
We can assume that these slide-turnings are chosen so  that $U_2'$ is strictly on the right of $t_2'$ but very close to it and, similarly, $U_1$ is strictly on the right of $t_1'$ but very close to it again. By $\langle$left, forward$\rangle$--$\langle$right, backward$\rangle$-duality, it suffices to show that the first slide-turning, which gives $t_2'$,  exists; 
the argument is the following. Let $U_2^\ast$ be the last point of $K\cap t_2$, see Figure~\ref{figprfrtcrsL}; it may coincide with $U_2$. Slide-turn $\pair{U_2}{t_2}$ around $K$ forward to
$\pair{U_2^\ast}{t_2}$ first.
 Since $U_2'\notin K$ comes after $U_2^\ast$, when  we slide-turn $\pair{U_2^\ast}{t_2}$ around $K$ a very little further, the line component of the pointed supported line we obtain will be a suitable $t_2'$. We can assume that the directions have changed so little that $t_2'$ and $t_1'$ are non-parallel and intersect on the left of the figure. In particular, if $\dir{t_1}=\dir{t_2}+\pi$, that is, $\alpha=\pi$, that causes no problem in the above argument.

\begin{figure}[ht] 
\centerline
{\includegraphics[scale=0.89]{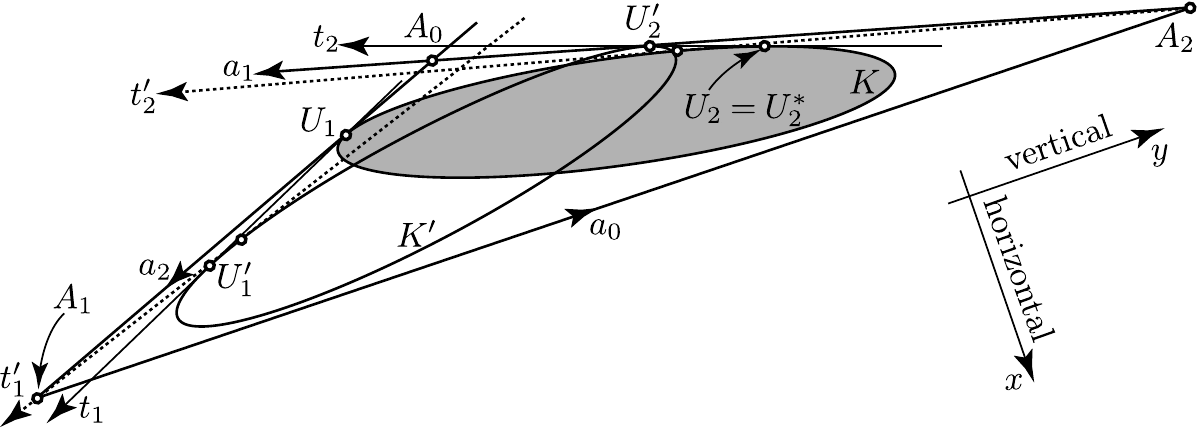}}
\caption{Another illustration for the proof of the Cross Lemma
\label{figcTngLcrL}}
\end{figure}%

Next, we slide-turn $\pair{U'_2}{t_2}$ forward around $K'$ to change the direction only a little;  let $a_1$ be the supporting line we obtain this way. Using that $K$ is bounded, it is on the left of $t_2$, and does not contain $U_2'$, it follows that if $\dir {a_1}-\dir{t_2}$ is small enough, then $K$ remains  on the left of $a_1$; furthermore and this is what we need, $K$ will be \emph{strictly}  on the left of $a_1$ in this case.   Hence, we can assume that $K$ is strictly on the left of $a_1$. Also, if $\dir {t_2'}-\dir{t_2}$ in the earlier slide-turning was small enough and now $\dir {a_1}-\dir{t_2}$ is even smaller, than $a_1$ and $t_2'$ intersect sufficiently far on the right at a point $A_2$. 
 Similarly, by slide-turning $\pair{U_1}{t_1}$  backward  around $K$ with changing the direction only a little, we obtain a supporting line $a_2$ of $K$ such that $K'$ is \emph{strictly} on the left of $a_2$. Again, we can assume that the intersection point $A_1:=a_2\cap t_1'$ is far on the right. Furthermore, continuity allows us to adjust the above-mentioned little quantities so that the directed line $a_0$ 
from $A_1$ to $A_2$ is of slope $\pi/2$ and, since it is sufficiently far, both $K$ and $K'$ are on the left of $a_0$.  Clearly, both $K$ and $K'$ are contained in the left halfplanes determined by $a_0$, $a_1$, and $a_2$. Since the intersection of these halfplanes is the triangle $\phullo{\set{A_0,A_1,A_2}}$, both  $K$ and $K'$ are contained in this triangle. Hence, to complete the proof by contradiction, we need to show that the conclusion part of \ref{thmmain}\eqref{thmmainc}, see at  \ref{thmmain}\eqref{thmmainb}, fails. Depending on $j\in\set{0,1,2}$ and $k\in\set{0,1}$ (that is, choosing between $K$ and $K'$), there are six cases.

First, if we slowly slide-turn $\pair{U_2^\ast}{t_2}$ around $K$ a little forward, then we arrive at a pointed supporting line whose line component, denoted by $t_2^\dagger$,  goes through $A_0$. Note that $t_2^\dagger$ is not  indicated in the figures. While $A_0$, $A_1$, and $K$ are on the left of 
$t_2^\dagger$, $U_2'\in K'$ is not. Hence, $K'\nsubseteq \phullo(K\cup \set{A_0,A_1})$.

Second, if we slide-turn $\pair {U_1}{t_1}$ around $K$ forward so that the direction changes only a very little and $t_1^\dagger$  denotes the line component of the pointed supporting line we obtain in this way, then $K$, $A_0$, and $A_2$ will be on the left of $t_1^\dagger$ but $U_1'\in K'$ will not. (Again, $t_1^\dagger$ is not drawn in the figures.) This shows that $K'\nsubseteq \phullo(K\cup \set{A_0,A_2})$.

Third, since $A_1$, $A_2$, and $K$ are on the left of $t_2'$ but $U'_2\in K'$ is not, we obtain that $K'\nsubseteq \phullo(K\cup \set{A_1,A_2})$. So far, we have shown that 
\begin{equation}K'\nsubseteq \phullo(K\cup (\set{A_0,A_1,A_2}\setminus\set{A_j}))\text{ for every }j\in\set{0,1,2}.
\label{eqjdmMMnbhFKv}
\end{equation}
Interchanging $\tuple{K,\text{forward},1,2}$ with
 $\tuple{K',\text{backward},2,1}$, we obtain that  
\begin{equation}K\nsubseteq \phullo(K'\cup (\set{A_0,A_1,A_2}\setminus\set{A_j}))\text{ for every }j\in\set{0,1,2}.
\label{eqjdmMsdjWnhbhS}
\end{equation}
Alternatively, slide-turn $\pair{U_2'}{t_2}$ around $K'$ backwards and use $U_2$ to check \eqref{eqjdmMsdjWnhbhS} for $j=2$, slide-turn $\pair{U_1'}{t_1}$ around $K'$ backwards and use $U_1$ for $j=1$, and use $t_1'$ and $U_1$ for $j=0$. The conjunction of \eqref{eqjdmMMnbhFKv} and \eqref{eqjdmMsdjWnhbhS} contradicts \eqref{eqtxtdhgrhcVStVn}. Thus, we have shown that $\dir{t_1}=\dir{t_2}+\alpha$ (mod $2\pi$) with $\alpha\in (0,\pi]$  is impossible.

\begin{figure}[ht] 
\centerline
{\includegraphics[scale=0.89]{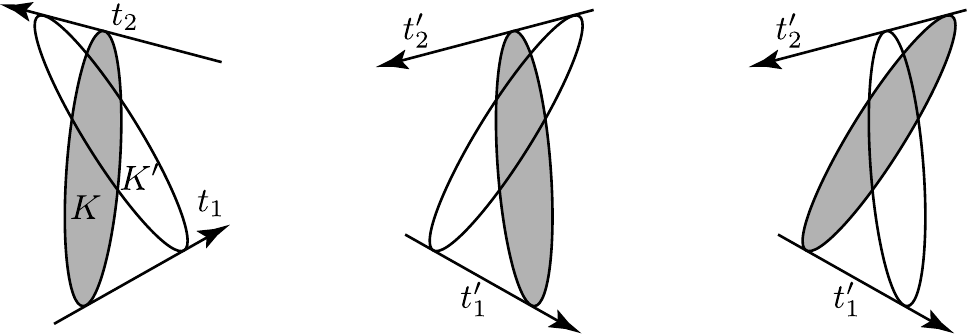}}
\caption{Reducing the case $\alpha\in (\pi,2\pi)$ to the case $\alpha\in (0,\pi]$
\label{fighnGTngLrDcS}}
\end{figure}%

Finally, if  $\dir{t_1}=\dir{t_2}+\alpha$ for some $\alpha\in (\pi,2\pi)$ modulo $2\pi$, then Figure~\ref{fighnGTngLrDcS} shows how to reduce this case to the already treated case  $\alpha\in (0,\pi]$. Namely, we reflect $K$, $K'$, $t_1$, and $t_2$  on the left of the figure across a vertical axis to obtain the middle part of the figure.  For $i\in\set{1,2}$, let $t_i'$ denote the mirror image of $t_i$ with opposite orientation; this is necessary to make it a common supporting line. In the next step, we interchange the roles of the mirror images of $K$ and $K'$; see on the right of the figure. We have arrived at the already treated case. Therefore, no matter what $\alpha\in(0,2\pi)$ is, \eqref{eqtxtdhgrhcVStVn} implies that $K$ and $K'$ cannot cross each other. This completes the proof of (the Cross) Lemma~\ref{lemmacross}. 
\end{proof}

\begin{proof}[Proof of \eqref{thmmainc} $\Rightarrow$ \eqref{thmmaina} for Theorem~\ref{thmmain}] 
Assuming \eqref{eqtxtdhgrhcVStVn}, we need to prove that $K$ is a disk.  
To do so, we are going to prove more and more ``disk-like'' properties of $K$ by contradiction, using the following technique: 
\begin{equation}
\parbox{11cm}{after supposing that $K$ fails to satisfy the given property, we show the existence of a $K'$ or $K^\ast$ such that $K$ and $K'$ (or $K^\ast$) are isometric and they cross each other, and (the Cross) Lemma~\ref{lemmacross} gives a contradiction.}
\label{eqtxttchnQe}
\end{equation}
To ease the terminology by using adjectives like ``left'', ``upper'', etc., we often assume that an arbitrary supporting line is \emph{horizontal} with direction $\pi$.
This does not affect the generality, because we can always choose a coordinate system appropriately.

\begin{figure}[ht] 
\centerline
{\includegraphics[scale=1.0]{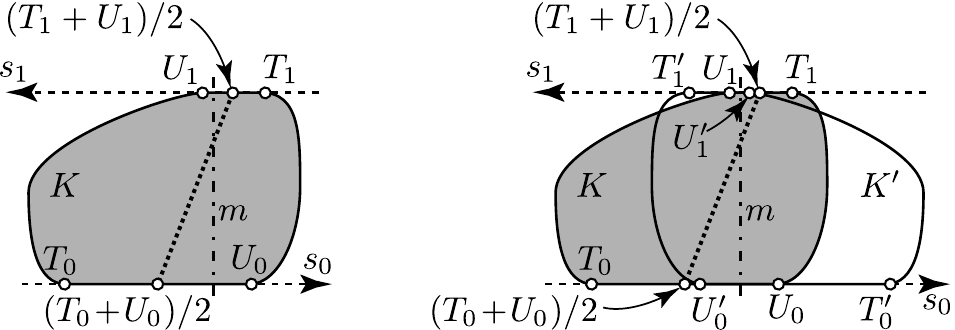}}
\caption{If \eqref{eqtxtoppsuplines} failed
\label{figslanted}}
\end{figure}%

Next, let $s_0$ and $s_1$ be two parallel supporting lines of $K$ with opposite orientation; 
see  Figure~\ref{figslanted}.  For $i\in\set{0,1}$, the intersection of $s_i$ and $K$ is a closed line segment $[T_i,U_i]\subseteq \bnd K$, possibly a singleton segment with $T_i=U_i$.  The \emph{middle point} of this line segment will be denoted by $(T_i+U_i)/2$. If the line through $(T_0+U_0)/2$ and $(T_1+U_1)/2$ is perpendicular to $s_0$, then we say that $s_0$ and $s_1$ are \emph{perpendicularly opposed}.  
We claim that 
\begin{equation}
\text{any two parallel  supporting lines of $K$ are  perpendicularly opposed.}
\label{eqtxtoppsuplines}
\end{equation}
Suppose, for a contradiction, that \eqref{eqtxtoppsuplines} fails; see Figure~\ref{figslanted}. Pick a line $m$ perpendicular to $s_0$ (and $s_1$) such that $(T_0+U_0)/2$ and $(T_1+U_1)/2$ are not on the same side of $m$. If we reflect $K$ across $m$ to obtain an isometric copy, $K'$, then $K$ and $K'$ cross each other; see on the right of Figure~\ref{figslanted}. By (the Cross) Lemma~\ref{lemmacross}, this is a contradiction that proves \eqref{eqtxtoppsuplines}.

\begin{figure}[ht] 
\centerline
{\includegraphics[scale=1.0]{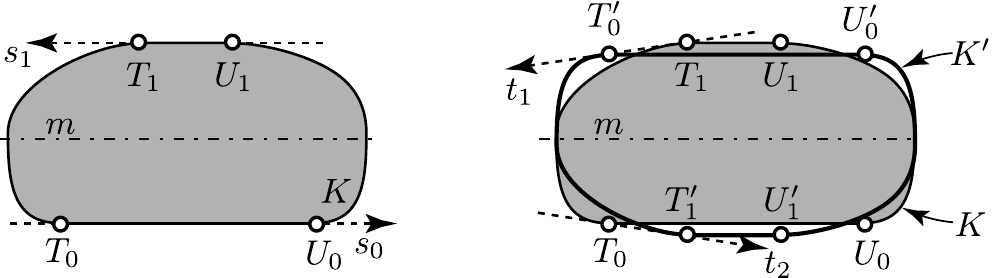}}
\caption{If \eqref{eqekvdSt} failed}
\label{figekvdSt} 
\end{figure}%

Next, using the notation given in the paragraph preceding \eqref{eqtxtoppsuplines}, we claim that 
\begin{equation}
\dist{T_0}{U_0} = \dist{T_1}{U_1}.
\label{eqekvdSt}
\end{equation}
If this fails, then pick a line $m$ parallel to and between $s_0$ and $s_1$ such $\dist m{s_0}$ is slightly smaller than $\dist m{s_1}$; see on the left of Figure~\ref{figekvdSt}. Reflecting $K$ across $m$, we obtain $K'$, see on the right; it follows from \eqref{eqtxtoppsuplines} that the situation is drawn in the figure correctly.  If $\dist m{s_1}-\dist m{s_0}$ is small enough, then the common supporting lines $t_1$ (obtained by slide-turning $\pair{T_1}{s_1}$ around $K$ a bit forward)  and $t_2$ (obtained by slide-turning $\pair{T_0}{s_0}$ around $K$ backward a little) of $K$ and $K'$  indicate that $K$ and $K'$ cross each other. Note that, say, $T_0$ need not belong to $t_2$, etc.. Since the situation contradicts Lemma~\ref{lemmacross}, we conclude~\eqref{eqekvdSt}.

\begin{figure}[ht] 
\centerline
{\includegraphics[scale=1.0]{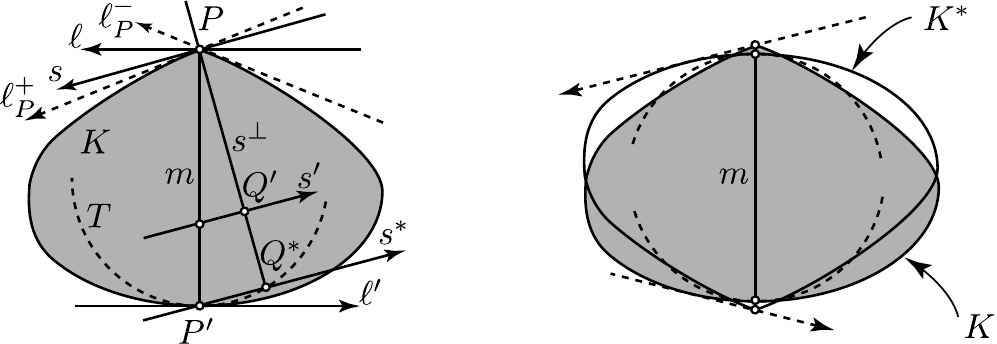}}
\caption{Corners lead to contradiction}
\label{figcrnRt} 
\end{figure}%

For $P\in\bnd K$, $\bnd K$ is \emph{smooth} at $P$ if $\flne P=\llne P$; see \eqref{eqtxtfllllnE}. If $\bnd K$ is not smooth at $P\in \bnd K$, then $P$ is a \emph{corner} of $K$ (and of $\bnd K$).  Next, we claim that  
\begin{equation}
\text{$\bnd K$ is smooth at each of its points.}
\label{eqtxtsmoothallpnts}
\end{equation}
Suppose, for a contradiction, that $P\in \bnd K$ such that  $\flne P\neq\llne P$; see Figure~\ref{figcrnRt} on the left.
Let $\lne$ be the supporting line through $P$ that halves the angle of $\flne P$ and $\llne P$; the figure is drawn such that $\dir\lne=\pi$. Clearly, $P$ is the only tangent point of $\lne$. 
Opposite to $\lne$, there is a unique supporting line $\lne'$ with $\dir{\lne'}=0$; it may have a full line segment of tangent points but let  $P'$ denote the middle one. We know from \eqref{eqtxtoppsuplines}
that the line $m$ through $P$ and $P'$ is vertical, that is, perpendicular to $\lne$. Let $T$ be the Thales circle of the line segment $[P,P']$; only a dashed circular arc of it is given in the figure.  Take an arbitrary supporting line $s$ through $P$; it is between $\flne P$ and $\llne P$.
Let $s'$ denote the supporting line parallel to $s$, and let $Q'$ be the middle tangent point on $s'$. Denoting the line through $P$ and $Q'$ by $s^\perp$, we know from \eqref{eqtxtoppsuplines} that $s^\perp$ is perpendicular to both $s$ and $s'$. Denote by  $s^\ast$ the directed line through $P'$ parallel to $s'$, and let $Q^\ast$ be the intersection point of $s^\ast$ and $s^\perp$.  Since  $s^\perp$ is perpendicular to $s^\ast$, $Q^\ast$ is on the Thales circle $T$. Clearly, if $\dist P{Q'}< \dist P{Q^\ast}$, then $P$ is on the right of the supporting line $s'$, which is impossible. Hence, $\dist P{Q'}\geq \dist P{Q^\ast}$, and we obtain that $Q'$ on $s^\perp$ is between $Q^\ast$ and $\lne'$. Hence, 
\begin{equation}
\text{in a neighborhood of $P'$, $\bnd K$ goes between $\lne'$ and $T$.}
\label{eqtxtlcsmcLkR}
\end{equation}
Since both $\lne'$ and $T$ are smooth at $P'$, where they touch each other, it follows that $\bnd K$ is smooth at $P'$. 
Finally, using that $K$ is smooth at $P'$ but ``acute'' at $P$, in particular, using \eqref{eqtxtlcsmcLkR}, it is easy to see the following:   if we reflect $K$ across a horizontal line only slightly closer to $\lne'$ than $\lne$, then we obtain an isometric copy $K'$ of $K$ such that $K$ and $K'$ cross each other; see on the right of Figure~\ref{figcrnRt}.  This contradicts (the Cross) Lemma~\ref{lemmacross} and proves \eqref{eqtxtsmoothallpnts}. 

\begin{figure}[ht] 
\centerline
{\includegraphics[scale=1.0]{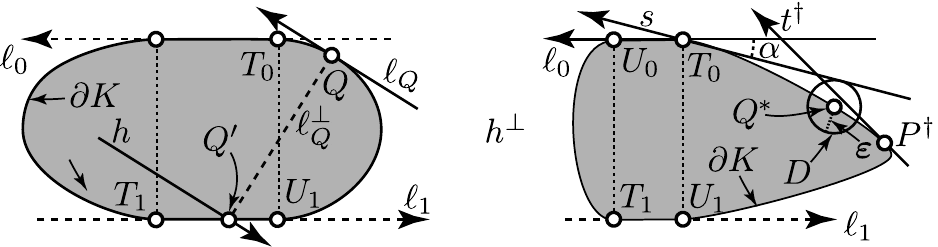}}
\caption{No ``tangent interval'' is possible}
\label{figngTZNmsQ} 
\end{figure}%

It follows from \eqref{eqtxtsmoothallpnts} that $\bnd K$ is everywhere smooth and   every supporting line of $K$ is a \emph{tangent line}; see  \eqref{eqtxtfllllnE}. In what follows, we will speak of tangent lines rather then supporting lines. So, for each $P\in\bnd K$, 
there is a unique \emph{tangent line} through $P$, and this tangent line is denoted by $\lne_P$.
Next, we claim that
\begin{equation}
\text{each tangent line of $K$ has exactly one tangent point.}
\label{eqtxtntGntRl}
\end{equation}
Suppose, for a contradiction, that \eqref{eqtxtntGntRl} fails. Let $\lne_0$ be a tangent line
for which  \eqref{eqtxtntGntRl} fails.  Then $\lne_0\cap K$ is a nontrivial line segment $[T_0,U_0]$. As usual, we can assume that $\dir{\lne_0}=\pi$.  Let $\lne_1$ be the tangent line parallel to $\lne_0$; see Figure~\ref{figngTZNmsQ}. It follows from \eqref{eqtxtoppsuplines} and \eqref{eqekvdSt} that the first and last tangent points on $\lne_0$ and $\lne_1$ form a rectangle $\tuple{T_0,U_0,T_1,U_1}$.

We know from \eqref{eqtxtsmoothallpnts} and Cz\'edli and Stach\'o~\cite{czgstacho} that \begin{equation}
\slcr K:=\set{\pair P{\dir{\lne_P}}: P\in\bnd K}\subseteq \real^4\text{  is a rectifiable Jordan curve.}
\label{eqtxtslcrrJc}
\end{equation}
This curve is the so-called \emph{slide curve} of $K$. Let $t^\dagger$ be the unique tangent line of $K$ with $\dir{t^\dagger}=3\pi/4$; see on the right of Figure~\ref{figngTZNmsQ}, and let $P^\dagger$ be the last point of $t^\dagger\cap \bnd K$. Denote by $B$ the arc of $\slcr K$ from $\pair{P^\dagger}{\dir{t^\dagger}}$ going forward to 
$\pair{T_0}{\dir{\lne_0}}$, and let $B^-:=B\setminus\set{\pair{P^\dagger}{3\pi/4}}$; the reader may want to but need not look into \cite{czgstacho} for details. (Our purpose with this arc is to make clear that Figure~\ref{figngTZNmsQ} on the right is more or less faithful.) The arc of $\bnd K$ from $P^\dagger$ to $T_0$ going forward will be denoted by $A$, and we let $A^-:=A\setminus\set{P^\dagger}$. 
We conclude by \cite[(4.9)--(4.11) and (4.16)]{czgstacho} that \begin{equation}
\text{if $\pair P{\dir{\lne_P}}\in B^-$, 
then $P\in A^-$.}
\label{eqtxtdhbnhmQsnW}
\end{equation}
Clearly, \eqref{eqtxtslcrrJc} allows us to take a  sequence $\vec s:=(\pair {P_n}{\dir{\lne_{P_n}}}: n\in\mathbb N)$ of elements of $B$ such that 
\begin{enumerate}
\item  $\pair{P_n}{\dir{\lne_{P_n}}}\neq\pair{T_0}{\dir{\lne_0}}$ and $\pair{P_n}{\dir{\lne_{P_n}}}\neq\pair{P^\dagger}{3\pi/4}$  for all $n\in\mathbb N$, 
\item $\pair{P_n}{\dir{\lne_{P_n}}}$ tends to  $\pair{T_0}{\dir{\lne_0}}=\pair{T_0}\pi$ as $n\to \infty$, and
\item\label{mzladglSd} $\dir{\lne_{P_n}}> 4\pi/5$ for all $n\in\mathbb N$.  
\end{enumerate}
Using the obvious inequalities 
\[
\begin{aligned}
\dist{\pair {X_1}{\dir{\lne_{X_1}}}}{\pair {X_2}{\dir{\lne_{X_2}}}}&\geq \dist {X_1}{X_2}\,\,\,\text{ and}\cr
\dist{\pair {X_1}{\dir{\lne_{X_1}}}}{\pair {X_2}{\dir{\lne_{X_2}}}}&\geq \dist{{\dir{\lne_{X_1}}}}{{\dir{\lne_{X_2}}}},
\end{aligned}
\]
we obtain that $P_n\to T_0$ and $\dir{\lne_{P_n}}\to \dir{\lne_0}=\pi$
 as $n\to\infty$. Furthermore, $P_n\in A^-$ by \eqref{eqtxtdhbnhmQsnW}.  Let $Q_n$ denote the middle point of the line section $\lne_{P_n}\cap \bnd K$. Since $\lne_{P_n}$ is a supporting line that contains $Q_n\in\bnd K$ and there is only one supporting line through $Q_n$ by  \eqref{eqtxtsmoothallpnts}, it follows that $\lne_{Q_n}=\lne_{P_n}$, whereby
\begin{equation}
\dir{\lne_{Q_n}}= \dir{\lne_{P_n}}\to \pi,\quad\text{as } n\to\infty.
\label{eqtxtasntoinfty}
\end{equation}
Note that for points $X,Y\in A$,  
\begin{equation}
\parbox{8.7cm}
{if we move   $X\in A$ slightly toward $Y\in A$ forward, then $\dir{\lne_X}$ moves toward to $\dir{\lne_Y}$ on  $\ucirc$ forward;}
\label{eqtxtmvFrWdnhB}
\end{equation}
this follows from a straightforward geometric consideration based on the fact that $X$ is on the left of $\lne_Y$ and $Y$ is on the left of $\lne_X$. Applying \eqref{eqtxtuNspalpha} to the intersection of $K$ and the right halfplane determined by the line from $P^\dagger$ to $T_0$, we obtain a point $P^\ddagger\in A\setminus\set{T_0,P^\dagger}$ such that  $\dir{\lne_{P^\ddagger}}=4\pi/5$. ($P^\ddagger$ is  between $P^\dagger$ and  $Q^\ast$ but it is not indicated in the figure.)  We can assume that $P^\ddagger$ is the first point of the line segment $\lne_{P^\ddagger}\cap \bnd K$. Thus, combining \eqref{mzladglSd}, \eqref{eqtxtasntoinfty}, and \eqref{eqtxtmvFrWdnhB}, it follows that
\begin{equation}
\text{$P_n$ and $Q_n$ are after $P^\ddagger$ on the arc $A^-$, for all $n\in\mathbb N$.}
\end{equation}

We claim that 
\begin{equation}
Q_n\to T_0\,\,\,\text{ as }\,\,\, n\to\infty.
\label{eqtxtQntndsTQ}
\end{equation}
Suppose, for a contradiction, that \eqref{eqtxtQntndsTQ} fails. Then 
the closed arc $A^\ddagger$ from $P^\ddagger$ to $T_0$ going forward, which is a subarc of $A$, is a compact set. Hence,  the sequence $(Q_n: n\in\mathbb N)$ has a limit point (also known as accumulation point) $Q^\ast\in A^\ddagger\setminus\set{T_0}$; see on the right of Figure~\ref{figngTZNmsQ}.
Thus, since $T_0$ is the first point of $\lne_0$, it follows easily from \eqref{eqtxtmvFrWdnhB} that $\pi>\dir{\lne_{Q^\ast}}\geq \dir{t^\dagger}=3\pi/4$. Combining this with the facts that $T_0$ is on the left of  $\lne_{Q^\ast}$ and  $Q^\ast$ is strictly below $\lne_0$, we obtain that $Q^\ast$ is strictly on the right of the directed line $\lne_{U_1,T_0}$ from $U_1$ to $T_0$. Therefore, denoting the disk of radius $\epsilon$ around $Q^\ast$ by $D=D(\epsilon)$, we can pick a small positive $\epsilon\in\real$ such that $D$ is strictly on the right of $\lne_{U_1,T_0}$, strictly on the left (that is, below) $\lne_0$, and $P^\dagger\notin D$.  (These stipulations ensure that Figure~\ref{figngTZNmsQ} on the right is appropriate to show the general case.) Let $s$ be the upper tangent line of $D$ through $T_0$, as indicated in the figure. Let $\alpha:=\pi-\dir s$; it is positive. Now if $Q_n\in D$, then $T_0\in K$ must be on the left of $\lne_{Q_n}$, whence $\dir{\lne_{Q_n}}<\pi-\alpha$. This happens for infinitely many $n\in \mathbb N$, which contradicts \eqref{eqtxtasntoinfty} and  proves \eqref{eqtxtQntndsTQ}.

Now, armed with \eqref{eqtxtasntoinfty} and \eqref{eqtxtQntndsTQ}, we can pick a point $Q\in\bnd K$ before $T_0$ such that 
$\pair Q{\dir{\lne_Q}}$ is ``sufficiently  close'' to $\pair {T_0}\pi$, to be specified soon, and $Q$ is the middle point of the line segment $\lne_Q\cap \bnd K$; see on the left of Figure~\ref{figngTZNmsQ}. Let $\lne^\perp_Q$ be the line through $Q$ such that $\dir{\lne^\perp_Q}=\dir{\lne_Q}+\pi/2$. Denoting the intersection point of $\lne^\perp_Q$ and $\lne_1$ by $Q'$, the term ``sufficiently close'' two sentences above means that $Q'$ is an interior point of the line segment $[T_1,U_1]$. It follows from  \eqref{eqtxtasntoinfty} and \eqref{eqtxtQntndsTQ} that this choice of $Q$ is possible. Clearly, $Q'\in\bnd H$. 
Let $h$ be the line through $Q'$ such that $h\parallel \lne_Q$ but they are directed oppositely. Then $h$ is perpendicular to $\lne^\perp_Q$, and  it follows from \eqref{eqtxtoppsuplines} that $h$ is a supporting line.  This is a contradiction, because the only supporting line through $Q'$ is $\lne_1$ but $\dir{h}=\dir{\lne_Q}+\pi\neq \pi+\pi \equiv 0=\dir{\lne_1}$ (mod $2\pi$). This contradiction completes the proof of \eqref{eqtxtntGntRl}.

From now on, we say that $P_1,P_2\in \bnd K$ are \emph{opposite} (\emph{boundary}) \emph{points} if $P_1\neq P_2$ and for both $i\in\set{1,2}$, the line through $P_i$ that is perpendicular to the line segment $[P_1,P_2]$ is a tangent line of $K$. In this case, the line segment $[P_1,P_2]$ is called a \emph{diagonal} of $K$.  Combining  \eqref{eqtxtfllllnE},   \eqref{eqtxtoppsuplines}, \eqref{eqtxtsmoothallpnts}, and \eqref{eqtxtntGntRl}, we obtain that
\begin{equation}
\parbox{9.5cm}{for each $P_1\in \bnd K$, there exists a unique $P_2\in \bnd H$ such that these two points are opposite. In other words, each point of $\bnd K$ is one of the endpoints of a unique diagonal of $K$.}
\label{eqtxtDGdFtn}
\end{equation}
Next, take an arc $A$ of $\bnd K$ such that $\lne_P$ is not vertical, that is, $\dir{\lne_P}\notin\set{\pi/2,3\pi/2}$  for all $P\in K$. 
%In Figure~\ref{figngTZprhlvdZs}, $A$ is the thick part of $\bnd K$.  
Also, take a coordinate system with  $x$-axis of direction 0; the position of the origin is irrelevant. It is easy see and the argument leading to Cz\'edli and Stach\'o~\cite[(4.3)]{czgstacho} clearly shows that $A$ is the graph of a unique real-valued function $p$. Furthermore, by \eqref{eqtxtsmoothallpnts} and (the argument leading to)  \cite[(4.6)]{czgstacho},
\begin{equation}
p\text{ is continuously differentiable in the interior of its domain.}
\label{eqtxtcontdif}
\end{equation}

\begin{figure}[ht] 
\centerline
{\includegraphics[scale=1.0]{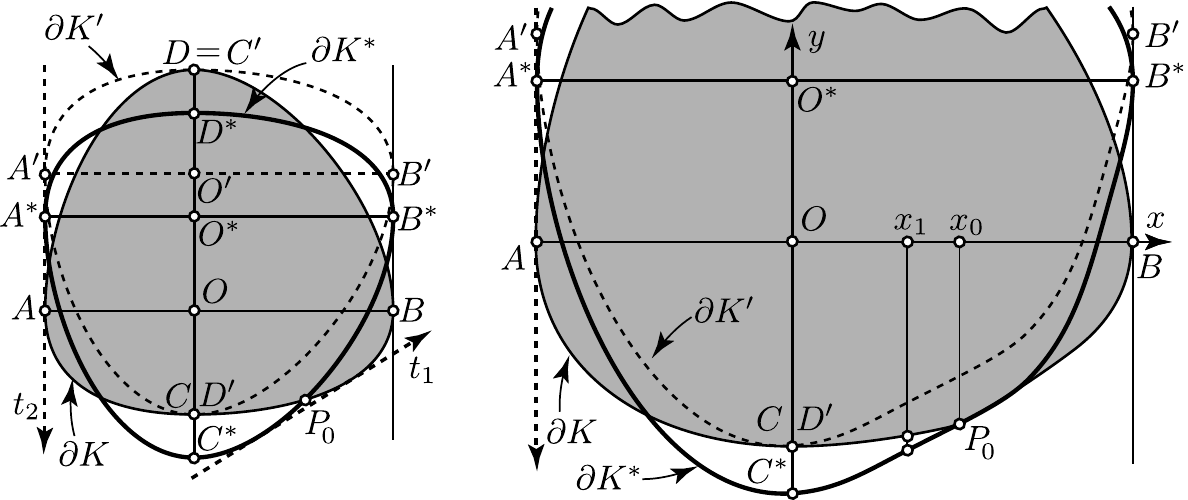}}
\caption{If \eqref{eqtxtperhalv} fails}
\label{figngRpjRpncZ} 
\end{figure}%

Next, we assert that 
\begin{equation}
\text{any two perpendicular diagonals of $K$ halve each other.}
\label{eqtxtperhalv}
\end{equation}
Suppose, for a contradiction, that this is false; 
see Figure~\ref{figngRpjRpncZ} on the left. Then there are perpendicular diagonals $[A,B]$ and $[C,D]$ of $K$ such that the first one does not halve the second.
Denoting their intersection point by $O$, we can assume that $\dist CO<\dist DO$. We choose a coordinate system according to the figure, that is, the origin is $O$, $B$ is on the positive part of the $x$-axis, and $D$ is above $O$. 
To obtain $K'$, we reflect $K$ to  the horizontal line that halves the distance $\dist CD$. 
The image of a point $X$ with respect to this reflection is denoted by $X'$; note that $C'=D$ and $D'=C$. 
Shift $K'$ down by a small $\epsilon\in\plreal$ to obtain $K^\ast$; the image of a point $X'$ by this shift is denoted by $X^\ast$. We can assume that $\epsilon$ is small enough to ensure that 
\begin{equation}
\text{$A^\ast$ and $B^\ast$ are in the interiors of $[A,A']$ and $[A,A']$, respectively.}
\label{eqtxtdhbnRmhfRl}
\end{equation} 
We focus on the south-eastern arcs of $\bnd K$ and $\bnd{K^\ast}$; see on the right of Figure~\ref{figngRpjRpncZ}. By \eqref{eqtxtcontdif}, these arcs can be defined by continuously differentiable real-valued functions $u$ and $u^\ast$, respectively. At $x=0$,  $C^\ast$ is below $C$, that is,   $u^\ast(0)<u(0)$, but later $B^\ast$ is above $B$. Hence, by continuity, there is a smallest $x_0$ where $u(x_0)=u^\ast(x_0)$, that is, where the two arcs \emph{meet} at a point,  which is denoted by $P_0$. 
The continuously differentiable function $v(x):=u^\ast(x)-u(x)$ is negative in $(0,x_0)$ and $v(x_0)=0$. We claim that 
\begin{equation}
\text{for every (small) $\delta\in (0,x_0)$,  $\exists x_1\in(x_0-\delta, x_0)$ with $v'(x_1)>0$.}
\label{eqtxtdzBmWxmbhW}
\end{equation}
In order to see this, suppose to the contrary that \eqref{eqtxtdzBmWxmbhW} fails. Then we have a $\delta\in (0,x_0)$ such that $v'(x)$ is nonpositive on $(x_0-\delta,x_0)$. Hence, the Newton--Leibniz rule yields that $v(x_0)-v(x_0-\delta)= \int_{x_0-\delta}^{x_0}v'(t)dt \leq 0$. Thus, $v(x_0-\delta)>v(x_0)=0$, contradicting the fact that $v(x)$ is negative in $(0,x_0)$. This proves \eqref{eqtxtdzBmWxmbhW}. 

\begin{figure}[ht] 
\centerline
{\includegraphics[scale=1.0]{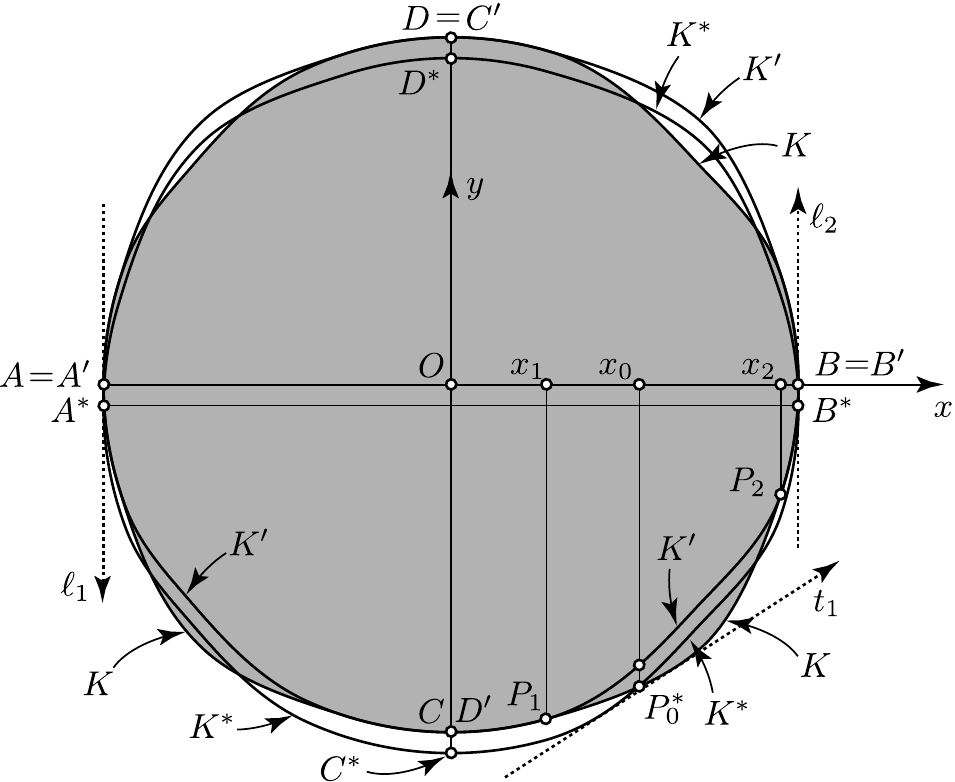}}
\caption{If \eqref{eqtxKsmwrtDgnls} fails}
\label{figjQRpxddhTZ} 
\end{figure}%

Next, since $v'$ is continuous at $x_0$ by \eqref{eqtxtcontdif}, it follows from  \eqref{eqtxtdzBmWxmbhW} that $v'(x_0) \geq 0$. Hence, there are two cases.
First, assume that $v'(x_0) > 0$, that is, ${u^\ast}{}'(x_0)>u'(x_0)$. This means that $\bnd K$ and $\bnd{K^\ast}$ cross each other at $P_0$ as indicated on the left of Figure~\ref{figngRpjRpncZ}, that is, their angle at $P_0$ is not 0. Since we can rotate the figure such that (the Intersection) Lemma~\ref{lemmaintersect} applies, we conclude that the common tangent line $t_1$, see on the left of the figure, touches $K^\ast$ before $K$. By \eqref{eqtxtdhbnRmhfRl}, $A^\ast$ is above $A$, whereby the common tangent line $t_2$ through them touches $K^\ast$ before $K$. Hence, (the Cross) Lemma~\ref{lemmacross} gives a contradiction. 
Second, assume that $u^\ast{}'(x_0)=u'(x_0)$; see on the right of Figure~\ref{figngRpjRpncZ}. Then \eqref{eqtxtdzBmWxmbhW} allows us to shift $K^\ast$ vertically upwards a little bit so that \eqref{eqtxtdhbnRmhfRl} remains valid and we obtain the first case with some $x_1$ instead of $x_0$. Hence, the second case leads to the same contradiction as the first one. This proves \eqref{eqtxtperhalv}. 

Next, we claim that 
\begin{equation}
\text{$K$ is symmetric with respect to each of its diagonals.}
\label{eqtxKsmwrtDgnls}
\end{equation}
In order to prove \eqref{eqtxKsmwrtDgnls}, let $[A,B]$ be a diagonal of $K$. Choosing the coordinate system appropriately, we can assume that this diagonal is horizontal and $A$ is to the left of $B$; see Figure~\ref{figjQRpxddhTZ}. As always, $K$ is grey-filled in the figure. We obtain $K'$ by reflecting $K$ across $[A,B]$; we need to show that $K'=K$. Actually, it suffices to show that $\bnd K=\bnd{K'}$. It suffices to deal with the lower right quarters of $\bnd K$ and $\bnd{K'}$, since the rest of the quarters can be treated similarly. Suppose, for a contradiction, that these lower right quarters are distinct. 
Denote by $[C,D]$ the diagonal of $K$ that is perpendicular to $[A,B]$, and let $O$ be the intersection point of the two diagonals. The notational conventions from the proof of \eqref{eqtxtperhalv} are still valid; for example, the mirror image of a point $X$ across the diagonal $[A,B]$ is denoted by $X'$; see  Figure~\ref{figjQRpxddhTZ}. Let $u_1$ and $u_2$ be the real-valued functions describing the lower right quarters of $\bnd K$ and $\bnd{K'}$. Since $C=D'$ by \eqref{eqtxtperhalv}, $u_1(0)=u_2(0)$. Let
\[x_1:=\sup\set{x': u_1(x)=u_2(x)\text{ for all }x\in[0,x']}. 
\]
Since $u_1$ and $u_2$  are continuously differentiable, see \eqref{eqtxtcontdif}, and, in particular, they are continuous, $u_1(x_1)=u_2(x_1)$. This situation is illustrated in Figure~\ref{figjQRpxddhTZ}. Note that, as opposed to the figure, we do not claim the existence of a \emph{first} $x_2$ such that $x_2>x_1$ and $u_1(x_2)=u_2(x_2)$. What we claim is that
\begin{equation}
\parbox{7.7cm}{in each right neighborhood of $x_1$, there exists an $x_0$ such that $u_1(x_0)\neq u_2(x_0)$ and $u_1'(x_0)\neq u_2'(x_0)$.}
\label{eqtxtdhnBpPlMqx}
\end{equation} 
In order to show this, define an auxiliary function $v$ by $v(x)=u_1(x)-u_2(x)$; this function is again continuously differentiable. Let $\epsilon\in\plreal$ be an arbitrary small number. Since  $x_1$ is defined as a supremum,  there exists an $x_3\in (x_1,x_1+\epsilon)$ such that $v(x_3)\neq 0$. By the continuity of $v$, we can take the largest $x_4\in[x_1,x_3]$ such that $v(x_4)=0$. Clearly, $x_1\leq x_4< x_3$ and $v$ has no root in the open interval $(x_4,x_3)$. 
By Lagrange's mean value theorem, there exists an $x_0\in(x_4,x_3)$ such that $v'(x_0)=(v(x_3)-v(x_4))/(x_3-x_4)=v(x_3)/(x_3-x_4) \neq 0$. Since $x_0\in(x_4,x_3)$, we also have that $v(x_0)\neq 0$, proving \eqref{eqtxtdhnBpPlMqx}.

Since $K$ and $K'$ play symmetric roles in our argument, we can assume that for $x_0$ in \eqref{eqtxtdhnBpPlMqx}, $v_1(x_0)<v_2(x_0)$. That is, at $x_0$, (the lower half of) $\bnd K$ is below $\bnd{K'}$.
Let $K^\ast$ denote what we obtain from $K'$ by shifting it vertically downwards by $v_2(x_0)-v_1(x_0)$. The intersection point of $\bnd K$ and $\bnd{K^\ast}$ with $x$-coordinate equal to $x_0$ will be denoted by $P_0^\ast$; see the figure. By the choice of $x_0$, $\bnd K$ and $\bnd{K^\ast}$ cross each other at $P_0^\ast$ with a nonzero angle. Let $t_1$ be the common tangent line provided by (the Intersection) Lemma~\ref{lemmaintersect}, applied either to $K$ and $K^\ast$, or to $K^\ast$ and $K$. 
Let $\lne_1$ and $\lne_2$ be the two vertical tangent lines of $K$. They are also tangent lines of $K^\ast$ and they are oppositely directed. 
Observe that $\lne_1$ touches $K$ first and $K^\ast$ only later but $\lne_2$ touches $K$ and $K^\ast$ in the reverse order.  Hence, we can pick a common tangent line $t_2\in\set{\lne_1,\lne_2}$ such that $t_1$ and $t_2$ yield a contradiction by (the Cross) Lemma~\ref{lemmacross}. This contradiction proves \eqref{eqtxKsmwrtDgnls}.

From now on, we pick two perpendicular diagonals $[A,B]$ and $[C,D]$, and let $O$ denote the point where they intersect; actually, where they halve each other. Let $\rho_1$ and $\rho_2$ be the axial reflections to the lines (determined by) $[A,B]$ and $[C,D]$, respectively. By \eqref{eqtxKsmwrtDgnls}, $K$ is invariant with respect to $\rho_1$ and also to $\rho_2$. Hence $K$ is invariant with respect to the composite map $\rho_1\circ\rho_2$, which is the central symmetry across $O$. This proves that
\begin{equation}
\text{$K$ is centrally symmetric and $O$ is the center of its symmetry.}
\label{eqtxtcTnrlsM}
\end{equation}

Next, we claim that 
\begin{equation}
\text{every diagonal of $K$ goes through $O$, the center  of symmetry.}
\label{eqtxtlldThGhH}
\end{equation}
In order to prove this, let $[X_1,X_2]$ be a diagonal of $K$. For $i\in\set{1,2}$, let $\lne_i$ denote the tangent line through $X_i$. By the definition of a diagonal, see around \eqref{eqtxtDGdFtn}, $\lne_1\parallel \lne_2$. Reflect  $X_1$ and $\lne_1$ across the point $O$; the point and line we obtain in this way are denoted by $X_1'$ and $\lne_1'$, respectively.
Also, let $K'$ denote the image of $K$ with respect to this point reflection. Clearly, $\lne_1'$ is tangent to $K'$ at $X_1'$. Actually,  $\lne_1'$ is tangent to $K$ at $X_1'$ since $K'=K$ by \eqref{eqtxtcTnrlsM}. Since $\lne_1'\parallel \lne_1$, we have three non-directed parallel tangent lines, $\lne_1$, $\lne_2$, and $\lne_1'$. By \eqref{eqtxtlwSnthGt}, these three lines cannot be distinct. Since
$\lne_1'$ is clearly distinct from $\lne_1$ and $\lne_2\neq\lne_1$, it follows that $\lne_1'=\lne_2$. Combining this with \eqref{eqtxtntGntRl}, we obtain that $X_2=X_1'$. Hence, $O\in [X_1,X_1']=[X_1,X_2]$, proving \eqref{eqtxtlldThGhH}.

Observe that in order to prove that $K$ is a disk, 
\begin{equation}
\parbox{7.4cm}{it suffices to show that each point  $P_0\in\bnd K$ has a neighborhood in which $\bnd K$ is a circular arc.}
\label{eqtxtdkhgnbdTsn}
\end{equation}
Indeed, in this case the compact set $\bnd K$ is covered by finitely many open sets, some of the neighborhoods above. We can assume that these open sets are pairwise incomparable with respect to set inclusion.  Arranging these open sets, which are circular arcs, cyclically counterclockwise according to their first limit points (outside them)  on $\bnd K$, we conclude that any two consecutive circular arcs  have a nonempty open intersection. So these two circular arcs have three (actually,  infinitely many) non-collinear points in common, which implies that these two circular arcs lie on the same circle. Therefore, all the finitely many circular arcs that cover $\bnd K$ lie on the same circle and $\bnd K$ is a circle. This shows the validity of \eqref{eqtxtdkhgnbdTsn}.

\begin{figure}[ht] 
\centerline
{\includegraphics[scale=1.0]{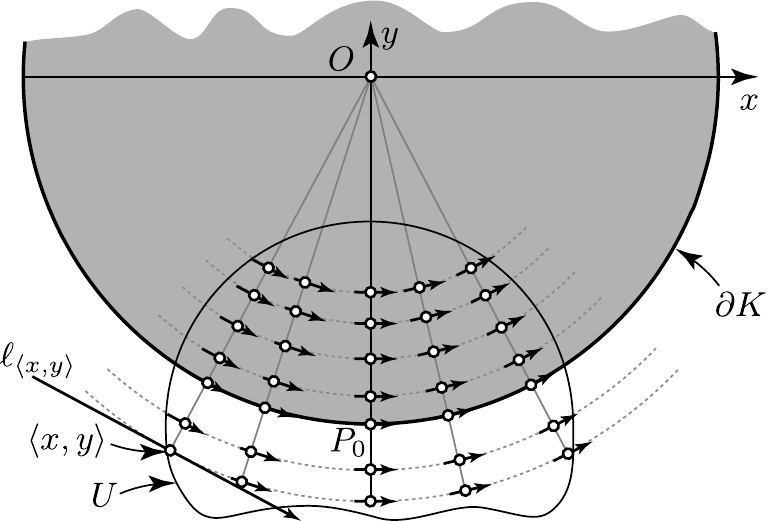}}
\caption{Possible tangent lines at $\pair xy\in U$}
\label{figjQvfHglDsZ} 
\end{figure}%

Next, combining \eqref{eqtxtoppsuplines}, \eqref{eqtxtsmoothallpnts}, \eqref{eqtxtntGntRl},  \eqref{eqtxtDGdFtn}, and \eqref{eqtxtlldThGhH},  we obtain that 
\begin{equation}
\parbox{7.6cm}{for every point $P_0\in\bnd K$, the unique tangent line $\lne_{P_0}$ is perpendicular to the line segment $[O,P_0]$.}
\label{eqtxtdpzRprdklr}
\end{equation}
Let $P_0$ be an arbitrary point of $\bnd K$. 
By choosing the coordinate system appropriately, we can assume that $\lne_{P_0}$ is horizontal and, in addition, $\dir{\lne_{P_0}}=0$; see Figure~\ref{figjQvfHglDsZ}. We assume that the coordinate axes $x$ and $y$ intersect at $O$, $x$ is horizontal and $P_0$ is below $x$; see the figure. We let $r_0=\dist O{P_0}$, that is, $P_0=\pair{0}{-r_0}$.
Let $U\subseteq \preal$ be a bounded open set such that $P_0\in U$, $U$ is below the $x$-axis, and the distance of $U$ and the $x$-axis is positive. For $\pair x y\in U$, let $\lne_{\pair x y}$ denote the unique line through $\pair x y$ such that $\lne_{\pair x y}$ is perpendicular to the line from $O$ through $\pair x y$.  We direct $\lne_{\pair x y}$ so that $O$ is on its left.  
In Figure~\ref{figjQvfHglDsZ},  $\lne_{\pair x y}$ is a long directed line for one choice of $\pair x y$, and it is a short directed line for several other choices of $\pair x y$. Let  $f(x,y)$ be the slope of $\lne_{\pair x y}$. It is easy to see that, for $\pair x y\in U$, $f(x,y)=-x/y$.   
Let $p_0$ be the real-valued function whose graph is the arc  $U\cap\bnd K$ of $\bnd K$. It is continuously differentiable by \eqref{eqtxtcontdif}. Hence, it follows from  \eqref{eqtxtdpzRprdklr}  that the function 
$p_0$ is a solution of the differential equation  
\begin{equation} p'(x)= f(x, p(x))\,\,\text{ with the initial condition }\,\, p(0)=y_0.
\label{eqtxtdffEq}
\end{equation}
Obviously, the circular arc $p_1(x):=-\sqrt{r_0^2-x^2}$ is a solution of \eqref{eqtxtdffEq} in a neighborhood of $0$. Since $U$ is strictly on the lower halfplane,   $f(x,y)=-x/y$ and its partial derivative, $\partial f/\partial y$ are continuous in $U$. Therefore,  to \eqref{eqtxtdffEq}, we can apply the well-known uniqueness theorem for differential equations; see, for example, King, Billingham, and Otto~\cite[Thm.\ 8.2 in page 211]{kingatal} or Ricardo~\cite[page 90]{ricardo}. In this way, we conclude that $p_1$ is the only solution of  \eqref{eqtxtdffEq} in a neighborhood of $0$. Thus, the two solutions, $p_0$ and $p_1$, are equal. Hence,
$\bnd K$ is a circular arc in a neighborhood of $P_0=\pair 0{p_0(0)}$. Therefore, since $P_0\in\bnd K$ was arbitrary,  \eqref{eqtxtdkhgnbdTsn} completes the proof.
\end{proof}

We conclude the paper with two remarks on the last part of the proof above.

\begin{remark} Instead of \eqref{eqtxtdkhgnbdTsn} and using compactness, it suffices to consider only three appropriately chosen points of $\bnd K$ to play the role of $P_0$ around \eqref{eqtxtdffEq}.
\end{remark}

\begin{remark} Instead of 
\eqref{eqtxtdffEq}, there is a bit more elementary but more computational argument. It runs as follows; we only outline it below. With respect to a coordinate system whose origin is $O$, we can write $\bnd K$ in the form $\set{\vec r(\alpha): \alpha\in \ucirc}$; see \eqref{equnitCrlcle}. In a standard way, it follows from \eqref{eqtxtcontdif} and the differentiability of compound functions  that $\vec r(\alpha)$ is everywhere differentiable. It is well known and easy to conclude that the slope of the tangent line $\set{\vec r(\alpha): \alpha\in \ucirc}$ through $\vec r(\alpha)$ is
\begin{equation}
\frac{\vvecr(\alpha)\sin \alpha  + \vec r(\alpha)\cos\alpha}{\vvecr(\alpha)\cos\alpha - \vec r(\alpha)\sin \alpha}\,.
\label{eqdhcnprdrVtvs}
\end{equation}
By \eqref{eqtxtdpzRprdklr}, 
the slope of this tangent line is $-\cos\alpha/\sin\alpha$. 
Comparing this with \eqref{eqdhcnprdrVtvs}, an easy calculation shows that $\vvecr(\alpha)=0$, for all $\alpha\in\ucirc$. Hence, $\vec r\,$ is a constant function and $\bnd K$ is a circle, as required.
\end{remark}

\subsection{Added on March, 9, 2017}
The referee called my attention to the fact that the following statement was implicit in the December 12, 2016 version of the paper. The meaning of ``cross each other" is given right after (the Cross) Lemma~\ref{lemmacross}.

\begin{lemma}\label{lemmalstmNt} Let $K_0$ be  a  compact convex subset of the plane $\preal$.
Then $K_0$ is a disk if and only if for every $K_1\subset \preal$ such that $K_1$ is isometric to $K_0$, the sets $K_0$ and $K_1$ do not cross each other.
\end{lemma}

If the interior of $K_0$ is nonempty, then the nontrivial direction of this lemma follows from \eqref{eqtxttchnQe}, while the statement is obvious if the interior of $K_0$ is empty.

In a forthcoming paper, we will use Lemma~\ref{lemmalstmNt} to establish a connection between the present paper and Fejes-T\'oth~\cite{fejestoth}. Finally, we note that our topic is also in connection with  a quite recent paper  by Kincses~\cite{kincses}.

\end{document}